\documentclass[12pt]{amsart}
\usepackage[all]{xy}
\usepackage[mathscr]{euscript}
\usepackage{hyperref}
\hypersetup{colorlinks=true,urlcolor=blue,citecolor=blue,linkcolor=blue}
\usepackage{courier}
\usepackage{amsmath,amscd,amssymb,amsthm,latexsym,longtable,diagram, picture}
\usepackage{graphicx}
\usepackage{array}
\usepackage{color}
\usepackage{enumerate}
\usepackage{nicefrac}
\usepackage{listings}
\usepackage{latexsym,bm,bbm,mathrsfs}
\usepackage{hyperref,graphicx, enumerate}
\usepackage{epsfig, xcolor, graphicx}
\usepackage[shortlabels]{enumitem}
\usepackage{indentfirst, setspace}
\usepackage{tikz-cd}

\allowdisplaybreaks
\usepackage{caption}
\usepackage{cancel}
\usepackage[normalem]{ulem}

\usetikzlibrary{calc,matrix,arrows,decorations.markings}

\newcommand{\bbz}{\mathbb{Z}}
\newcommand{\bbq}{\mathbb{Q}}

\newcommand{\bbc}{\mathbb{C}}
\newcommand{\bbr}{\mathbb{R}}

\newcommand{\mat}{\begin{pmatrix}}
\newcommand{\emat}{\end{pmatrix}}
\newcommand{\rank}{\mathrm{rank}}
\newtheorem{theorem}{Theorem}[section]
\newtheorem{proposition}[theorem]{Proposition}
\newtheorem{corollary}[theorem]{Corollary}
\newtheorem{lemma}[theorem]{Lemma}
\newtheorem{rmk}[theorem]{Remark}

\begin{document}
 \title[Number of integral points on quadratic twists of elliptic curves]
{Number of integral points on quadratic twists of elliptic curves}

\author{Seokhyun Choi}

%\author[S. Choi]{Seokhyun Choi}
\address{
Dept. of Mathematical Sciences, KAIST,
291 Daehak-ro, Yuseong-gu,
Daejeon 34141, South Korea
}
\email{sh021217@kaist.ac.kr}

\date{\today}

\begin{abstract}
    We study integral points on the quadratic twists $E_D : y^2 = x^3+D^2Ax+D^3B$ of a fixed elliptic curve $E : y^2 = x^3+Ax+B$ over $\bbq$. For sufficiently large squarefree positive integers $D$, we prove that the number of integral points on $E_D$ admits the upper bound $\ll 4^r$, where $r$ denotes the Mordell-Weil rank of $E_D$. The implied constant is absolute and effectively computable. The proof combines gap principles, bounds for spherical codes, and Diophantine approximation. As an application, we prove that the average number of integral points on the quadratic twist family is bounded.
\end{abstract}

\maketitle

\section{Introduction}

Let $E/\bbq$ be an elliptic curve defined by the Weierstrass equation 
\[E:y^2 = x^3+Ax+B,\quad A,B \in \bbz.\]
In 1929, Siegel \cite{Sie29} proved that the set of integral points $E(\bbz)$ is finite. The principal tool in the argument was the Thue-Siegel-Roth theorem in Diophantine approximation. One notable limitation is that, due to the ineffectivity of the Thue-Siegel-Roth theorem, Siegel's theorem on integral points is itself also ineffective.

Baker \cite{Bak22} was the first to obtain an effective bound for the heights of integral points. He proved that if $(x,y) \in E(\bbz)$, then 
\[\lvert x \rvert \leq \exp\left((10^6\max\{\lvert A \rvert,\lvert B \rvert\})^{10^6}\right).\] 
Lang \cite{Lan83} conjectured that there exist absolute constants $C$ and $\kappa$ such that if $(x,y) \in E(\bbz)$, then 
\[\lvert x \rvert \leq C(\max\{\lvert A \rvert,\lvert B \rvert\})^\kappa.\]
However, very little is known about this conjecture.

A related, but more tractable problem is to establish an upper bound for the number of integral points. Lang \cite{Lan78} conjectured that there exists an absolute constant $C$ such that 
\[\lvert E(\bbz) \rvert \ll C^r\]
where $r$ is the rank of $E/\bbq$. Silverman \cite{Sil87} proved Lang's conjecture for elliptic curves with $j$-invariant non-integral for bounded number of primes. Hindry and Silverman \cite{HS88} proved Lang's conjecture for elliptic curves with bounded Szpiro ratio. Helfgott and Venkatesh \cite{HV06} proved that there exists an absolute constant $C$ such that 
\[\lvert E(\bbz) \rvert \ll C^{\omega(\Delta)}(\log \lvert \Delta \rvert)^2(1.34)^r\]
where $\Delta$ is the discriminant of $E/\bbq$ and $r$ is the rank of $E/\bbq$.

We now restrict our attention to the quadratic twist family of a fixed elliptic curve. For each squarefree positive integer $D$, consider the quadratic twist of $E$ by $D$, 
\[E_D:y^2 = x^3+D^2Ax + D^3B.\]
Quadratic twist families are a natural setting for studying uniform bounds, as they depend on a single parameter $D$.

In this setting, one seeks bounds for $\lvert E_D(\bbz)\rvert$ that depend only on the Mordell-Weil rank of $E_D/\bbq$. By \cite{Sil87}, there exists an absolute constant $C$ such that 
\[\lvert E_D(\bbz) \rvert \ll_{A,B} C^r\]
where $r$ is the Mordell-Weil rank of $E_D/\bbq$. The constant $C$ is computable, though is was not given in explicit form. Gross and Silverman \cite{GS95} were the first to provide an explicit value of $C$, which is of order $10^9$. This was subsequently improved by Chi, Lai, and Tan \cite{CLT05}, who refined the constant $C$ to about $25$. More recently, Chan \cite{Cha22} proved the sharper bound $C=3.8$ for the congruent number curve.

A principal motivation of this paper is to extend Chan's result for the congruent number curve to general quadratic twist families. We prove that much of Chan's method carries over to this setting, giving a bound with constant $C=4$. This improves the previous general bound of Chi, Lai, and Tan.

\begin{theorem}\label{main_theorem}
    Let $E/\bbq$ be an elliptic curve defined by the Weierstrass equation 
    \[E:y^2 = x^3+Ax+B,\quad A,B \in \bbz.\]
    For any squarefree positive integer $D$, let $E_D$ be the quadratic twist of $E$ by $D$, defined by the Weierstrass equation 
    \[E_D:y^2 = x^3+D^2Ax + D^3B.\]
    Then there exists a constant $D(A,B)$, depending only on $A,B$, such that 
    \[\left\lvert E_D(\bbz) \right\rvert \ll 4^{\rank(E_D/\bbq)}\quad \text{whenever}\quad D \geq D(A,B).\]
\end{theorem}

\begin{rmk}
    The bound $\ll$ is an absolute constant which is effectively computable.
\end{rmk}

Theorem~\ref{main_theorem} also has a natural arithmetic consequence. Recent work of Smith \cite{Smi22} shows that exponential moments of the Mordell-Weil rank are uniformly bounded in quadratic twist families. Specializing \cite[Theorem~1.1]{Smi22} to our setting, we obtain the following theorem.

\begin{theorem}\label{Smith_theorem}
    Let $E/\bbq$ be an elliptic curve. Let $N(X)$ be the number of squarefree positive integers $D$ such that $0<D \leq X$. Then there exists a constant $C>0$ depending only on $E/\bbq$ such that 
    \[\frac{1}{N(X)}\sum_{\substack{0< D \leq X \\ D\text{ squarefree}}} 4^{\rank{(E_D/\bbq)}} \leq C,\quad X>0.\]
\end{theorem}

Applying Theorem~\ref{main_theorem} and Theorem~\ref{Smith_theorem}, we can prove that the average number of integral points on the quadratic twist family is bounded.

\begin{corollary}\label{average_corollary}
    Let $E/\bbq$ and $E_D/\bbq$ be as in Theorem~\ref{main_theorem}. Let $N(X)$ be the number of squarefree positive integers $D$ such that $0<D \leq X$. Then 
    \[\limsup_{X \rightarrow \infty} \frac{1}{N(X)}\sum_{\substack{0<D \leq X \\ D\text{ squarefree}}} \left\lvert E_D(\bbz) \right\rvert \ll C,\]
    where $C$ is a constant appearing in Theorem~\ref{Smith_theorem}.
\end{corollary}
\begin{proof}
    Define 
    \[S := \sum_{\substack{0<D \leq D(A,B) \\ D\text{ squarefree}}} \left\lvert E_D(\bbz) \right\rvert.\]
    Then 
    \[\frac{1}{N(X)}\sum_{\substack{0<D \leq X \\ D\text{ squarefree}}} \left\lvert E_D(\bbz) \right\rvert \ll \frac{S}{N(X)}+ \frac{1}{N(X)} \sum_{\substack{D(A,B)<D \leq X \\ D\text{ squarefree}}} 4^{\rank{(E_D/\bbq)}}\leq \frac{S}{N(X)} + C.\]
    Letting $X \rightarrow \infty$ proves the theorem.
\end{proof}

\begin{rmk}
    The bound $\ll$ in Corollary~\ref{average_corollary} is absolute. Consequently, if the constant $C$ in Theorem~\ref{Smith_theorem} could be chosen independently of $E/\bbq$, then the same would hold for Corollary~\ref{average_corollary}. In particular, this implies that the average number of integral points on the quadratic twist family is uniformly bounded.
\end{rmk}

\begin{rmk}
    Recently, Browning and Chan \cite{BC24} proved a much stronger average result. They proved that the average number of integral points on the quadratic twist family tends to $0$, assuming the weak Hall-Lang conjecture if $x^3+Ax+B$ has precisely one root over $\bbq$. 
\end{rmk}

The proof of Theorem~\ref{main_theorem} combines gap principles, bounds for spherical codes, and Diophantine approximation, but these ingredients enter in different ways according to the height of the integral point. We begin by partitioning $E_D(\bbz)$ into four subsets:
\begin{gather*}
    E_D(\bbz)_{small} := \{P \in E_D(\bbz)\:|\:\hat{h}(P) \leq 1.5\log D\}, \\
    E_D(\bbz)_{medium-small} := \{P \in E_D(\bbz)\:|\:1.5\log D \leq \hat{h}(P) \leq 20\log D\}, \\
    E_D(\bbz)_{medium-large} := \{P \in E_D(\bbz)\:|\:20\log D \leq \hat{h}(P) \leq 2200\log D\}, \\
    E_D(\bbz)_{large} := \{P \in E_D(\bbz)\:|\:\hat{h}(P) \geq 2200\log D\}.
\end{gather*}

To bound $E_D(\bbz)_{small}$, $E_D(\bbz)_{medium-small}$, and $E_D(\bbz)_{medium-large}$, we establish gap principles between integral points, and then apply bounds from spherical codes. The form of the gap principle used in our argument was originally formulated by Helfgott \cite{Hel04}, and subsequently refined by Alpoge \cite{Alp14}.

To bound $E_D(\bbz)_{large}$, we once again employ the gap principle, but we will also employ the techniques from Diophantine approximation. Roughly speaking, we prove that large integral points yield rational approximations to a fixed algebraic number with exponent $>2$. We then invoke the quantitative Roth's theorem to obtain a bound of large integral points. The Diophantine approximation of this kind originates from Alpoge \cite{Alp14}, and subsequently refined by Chan \cite{Cha22}.

\section{Notations}

In this paper, we fix an elliptic curve $E/\bbq$ defined by the Weierstrass equation 
\[E:y^2 = x^3+Ax+B,\quad A,B \in \bbz.\]
Define $\Delta = -16(4A^3+27B^2)$ and $j = -1728\frac{(4A)^3}{\Delta}$. We also fix a positive integer $M$ satisfying $\max\{10\sqrt{\lvert A \rvert},5\sqrt[3]{\lvert B \rvert}\} \leq M$. 

Given a squarefree positive integer $D$, let $E_D$ be the elliptic curve defined by 
\[E_D:y^2 = x^3+D^2Ax + D^3B\]
and let $E^D$ be the elliptic curve defined by 
\[E^D : Dy^2 = x^3+Ax+B.\]
The two elliptic curves $E_D$ and $E^D$ are quadratic twist of $E$ by $D$, and isomorphic over $\bbq$ by the isomorphism 
\[\phi_D:E_D \longrightarrow E^D,\quad (x,y) \longmapsto (x/D,y/D^2).\]

When we say "for sufficiently large $D$", it means $D \geq D(A,B)$ for some squarefree positive integer $D(A,B)$ depending only on $A,B$. When we use the constants $c_1,c_2,\ldots$, they depend only on $A,B$. When we write $f \ll g$, we mean $\lvert f \rvert \leq C \lvert g \rvert$, where $C$ is an absolute constant.

Let $h$ be the absolute logarithmic height function on $\overline{\bbq}$. For an elliptic curve $E/\bbq$ and a point $P \in E(\overline{\bbq})$, define 
\[h(P) = h(x(P)) \qquad \text{and} \qquad\hat{h}(P) = \lim_{n \rightarrow \infty} \frac{h(2^nP)}{4^n}.\]
Note that $\hat{h}$ is not normalized by the factor $\frac{1}{2}$.

\section{Preliminary lemmas}

We begin by comparing the canonical height $\hat{h}$ and the Weil height $h$ on $E$.

\begin{lemma}\label{Weil_canonical_height_difference}
    Let $P \in E(\overline{\bbq})$. Then 
    \[c_1 \leq \hat{h}(P) - h(P) \leq c_2.\]
\end{lemma}
\begin{proof}
    By \cite{Sil90}, 
    \[-\frac{1}{4}h(j)-1.946 - \frac{1}{6}h(\Delta) \leq \hat{h}(P) - h(P) \leq \frac{1}{6}h(j) + 2.14 + \frac{1}{6}h(\Delta).\]
\end{proof}

Next, we estimate heights on $E^D$.

\begin{lemma}\label{heights_for_E^D}
    Let $P \in E^D(\bbq)$. Then 
    \begin{equation}\label{Weil_canonical_1}
        c_1 \leq \hat{h}(P) - h(P) \leq c_2
    \end{equation}
    and if $P$ is non-torsion, 
    \begin{equation}\label{canonical_lower_bound_1}
        \hat{h}(P) \geq \frac{1}{4}\log D + c_3.
    \end{equation}
\end{lemma}
\begin{proof}
    Apply Lemma~\ref{Weil_canonical_height_difference} and \cite[Lemma~4.1]{Hel04} for \eqref{Weil_canonical_1} and \cite[Corollary~4.3]{Hel04} for \eqref{canonical_lower_bound_1}.
\end{proof}

From the isomorphism $\phi_D:E_D \rightarrow E^D$, we can estimate heights on $E_D$.

\begin{lemma}\label{Weil_canonical_height}
    Let $P \in E_D(\bbq)$. Then 
    \begin{equation}\label{Weil_canonical_2}
        c_1 - \log D \leq \hat{h}(P) - h(P) \leq c_2 + \log D.
    \end{equation}
    If in addition $x(P)>D$, then 
    \begin{equation}\label{Weil_canonical_3}
        c_1 - \log D \leq \hat{h}(P) - h(P) \leq c_2
    \end{equation}
\end{lemma}
\begin{proof}
    Since $\phi_D$ is an isomorphism, $\hat{h}(P) = \hat{h}(\phi_D(P))$. Next, $h(\phi_D(P)) = h(x(P)/D)$ and $h(P) = h(x(P))$. We have 
    \[h(x(P)) - \log D \leq h(x(P)/D) \leq h(x(P)) + \log D,\]
    and if in addition $x(P)>D$, then 
    \[h(x(P)/D) \leq h(x(P)).\]
    By applying \eqref{Weil_canonical_1}, we obtain \eqref{Weil_canonical_2} and \eqref{Weil_canonical_3}.
\end{proof}

\begin{lemma}\label{canonical_height_lower_bound}
    Let $P \in E_D(\bbq)$. If $P$ is non-torsion, then 
    \begin{equation}\label{canonical_lower_bound_2}
        \hat{h}(P) \geq \frac{1}{4}\log D + c_3.
    \end{equation}
\end{lemma}
\begin{proof}
    Since $\phi_D$ is an isomorphism, $\hat{h}(P) = \hat{h}(\phi_D(P))$. Applying \eqref{canonical_lower_bound_1} gives \eqref{canonical_lower_bound_2}.
\end{proof}

The next lemma concerns about torsion subgroups of $E_D$.

\begin{lemma}\label{torsion}
    For sufficiently large $D$, 
    \[E_D(\bbq)_{tors} \in \{0,\bbz/2\bbz,(\bbz/2\bbz)^2\}.\]
\end{lemma}
\begin{proof}
    By \cite[Lemma~4.2]{Hel04}, $E^D(\bbq)_{tors} \in \{0,\bbz/2\bbz,(\bbz/2\bbz)^2\}$ for sufficiently large $D$. Since $E^D$ and $E_D$ are isomorphic over $\bbq$, the lemma is proved.
\end{proof}

We will now estimate $x(P+Q)$ for $P,Q \in E_D(\bbq)$.

\begin{lemma}\label{P+Q_positive_bound}
    Let $P,Q \in E_D(\bbq)$ and $MD \leq x(P)<x(Q)$. Suppose $y(P)y(Q)>0$. Then 
    \[0.19x(P) \leq x(P+Q) \leq 2x(P).\]
\end{lemma}
\begin{proof}
    Recall that 
    \[x(P+Q) = \frac{(x(P)x(Q)+D^2A)(x(P)+x(Q))+2D^3B - 2y(P)y(Q)}{(x(P)-x(Q))^2}.\]
    Let $\lambda = x(Q)/x(P)$, $a=D^2A/x(P)^2$, $b=D^3B/x(P)^3$. Then 
    \[x(P+Q) = \frac{(\lambda+a)(\lambda+1)+2b-2\sqrt{\lambda^3+a\lambda+b}\sqrt{1+a+b}}{(1-\lambda)^2}x(P).\]
    By our choice of $M$, $\lvert a \rvert,\lvert b \rvert \leq 0.01$. Thus by Lemma~\ref{f(x)_lower_bound} and Lemma~\ref{f(x)_upper_bound}, 
    \[0.19x(P) \leq x(P+Q) \leq 2x(P).\]
\end{proof}

\begin{lemma}\label{P+Q_negative_bound}
    Let $P,Q \in E_D(\bbq)$ and $MD \leq x(P)<x(Q)$. Suppose $y(P)y(Q)<0$. Then 
    \[x(P) \leq x(P+Q) \leq \frac{(2\mu+1)^2}{(\mu-1)^2}x(P)\]
    where $\mu \leq x(Q)/x(P)$
\end{lemma}
\begin{proof}
    Recall that 
    \[x(P+Q) = \frac{(x(P)x(Q)+D^2A)(x(P)+x(Q))+2D^3B - 2y(P)y(Q)}{(x(P)-x(Q))^2}.\]
    Let $\lambda = x(Q)/x(P)$, $a=D^2A/x(P)^2$, $b=D^3B/x(P)^3$. Then 
    \[x(P+Q) = \frac{(\lambda+a)(\lambda+1)+2b+2\sqrt{\lambda^3+a\lambda+b}\sqrt{1+a+b}}{(1-\lambda)^2}x(P).\]
    By our choice of $M$, $\lvert a \rvert,\lvert b \rvert \leq 0.01$. Thus by Lemma~\ref{g(x)_lower_bound} and Lemma~\ref{g(x)_upper_bound}, 
    \[x(P) \leq x(P+Q) \leq \frac{(2\lambda+1)^2}{(\lambda-1)^2}x(P).\]
    Since $(2x+1)^2/(x-1)^2$ is strictly decreasing for $x>1$, we have 
    \[\frac{(2\lambda+1)^2}{(\lambda-1)^2} \leq \frac{(2\mu+1)^2}{(\mu-1)^2}.\]
\end{proof}

We next estimate $x(3P)$ for $P \in E_D(\bbq)$.

\begin{lemma}\label{3P_bound}
    Let $P \in E_D(\bbq)$ and $MD \leq x(P)$. Then 
    \[0.01x(P) \leq x(3P) \leq 0.27x(P).\]
\end{lemma}
\begin{proof}
    We have 
    \[x(3P) = \frac{\phi_3(P)}{\psi_3(P)^2}\]
    where $\psi_n(x)$ is an $n$-th division polynomial and $\phi_n = x\psi_n(x)^2-\psi_{n+1}(s)\psi_{n-1}(x)$. Thus  
    \[\psi_3(x) = 3x^4+6D^2Ax^2+12D^3Bx-D^4A^2\]
    and 
    \begin{align*}
        \phi_3(x) = &x^9 - 12D^2Ax^7-96D^3Bx^6+30D^4A^2x^5-24D^5ABx^4+D^6(36A^3+48B^2)x^3 \\
        &+ 48D^7A^2Bx^2+D^8(9A^4+96AB^2)x+D^9(8A^3B+64B^3).
    \end{align*}
    Let $a=D^2A/x(P)$, $b=D^3B/x(P)$. Then 
    \[x(3P) = \frac{c}{(3+6a+12b-a^2)^2} x(P)\]
    where 
    \begin{align*}
        c=&1-12a-96b+30a^2-24ab+(36a^3+48b^2)\\
        &+48a^2b+(9a^4+96ab^2)+(8a^3b+64b^3).
    \end{align*}
    By our assumption, $\lvert a \rvert \leq 0.01$, $\lvert b \rvert \leq 0.008$. Then some calculations show 
    \[0.1 \leq c \leq 2.1\]
    and 
    \[2.8 \leq 3+6a+12b-a^2 \leq 3.156.\]
    Thus 
    \[0.01 \leq \frac{0.1}{3.156^2} \leq\frac{c}{(3+6a+12b-a^2)^2} \leq \frac{2.1}{2.8^2} \leq 0.27.\]
    and hence 
    \[0.01x(P) \leq x(3P) \leq 0.27x(P).\]
\end{proof}

Finally, we estimate heights of integral points on $E_D$. This lemma implies the gap principle we will use later. 

\begin{lemma}\label{Weil_height_estimate}
    Let $P,Q$ be integral points on $E_D$ satisfying $MD \leq x(P)<x(Q)$. Then  
    \[h(P+Q) \leq h(P)+2h(Q)+2.9.\]
\end{lemma}
\begin{proof}
    We have 
    \begin{align*}
        x(P+Q) &= \left(\frac{y(P)-y(Q)}{x(P)-x(Q)}\right)^2 - (x(P)+x(Q)) \\
        &= \frac{y(P)^2 + y(Q)^2 - 2y(P)y(Q) - x(P)^3-x(Q)^3+x(P)x(Q)(x(P)+x(Q))}{(x(P)-x(Q))^2} \\
        &= \frac{(x(P)x(Q)+D^2A)(x(P)+x(Q))+2D^3B - 2y(P)y(Q)}{(x(P)-x(Q))^2}.
    \end{align*}
    By using the estimates 
    \[h(x+y) \leq \max\{h(x),h(y)\} + \log 2,\quad h(xy) \leq h(x)+h(y),\]
    and the choice of $M$, we have 
    \begin{gather*}
        h((x(P)x(Q)+D^2A)(x(P)+x(Q))) \leq h(x(P))+2h(x(Q))+\log 6, \\
        h(2D^3B) \leq h(x(P))+2h(x(Q))+\log 6, \\
        h(2y(P)y(Q)) \leq h(x(P))+2h(x(Q))+\log 6.
    \end{gather*}
    Therefore, we have 
    \[h((x(P)x(Q)+D^2A)(x(P)+x(Q))+2D^3B - 2y(P)y(Q))) \leq h(x(P))+2h(x(Q))+\log 18.\]
    Since $x(P)<x(Q)$, 
    \[2h(x(P)-x(Q)) \leq 2h(x(Q)) \leq h(x(P))+2h(x(Q))+\log 18.\]
    Hence, 
    \[h(x(P+Q)) \leq h(x(P))+2h(x(Q))+\log 18.\]
\end{proof}

\section{Spherical codes}

Let $r$ be a positive integer and $0<\theta<2\pi$ be an angle. Let $\Omega_r$ be the unit sphere in $\bbr^r$ and let $X$ be a finite subset of $\Omega_r$. We call $X$ a spherical code if $\langle x,y \rangle \leq \cos \theta$ for every distinct $x,y \in X$. We write $A(r,\theta)$ for the maximum size of the spherical code $X$.

As explained below, we must establish a bound for $A(r,\theta)$ with fixed $0<\theta<2\pi$. There are many bounds for $A(r,\theta)$ (see \cite{CS99}[Chapter~1]). Among them, we shall make use of only two bounds; one for $0<\theta<\pi/2$ and one for $\theta>\pi/2$. When $0<\theta<\pi/2$, we will use the bound given by Kabatiansky and Levenshtein.

\begin{theorem}\label{spherical_code_bound1}
    For fixed $0<\theta<\pi/2$, 
    \[\frac{1}{r}\log A(r,\theta) \leq \frac{1+\sin \theta}{2\sin \theta}\log\frac{1+\sin \theta}{2\sin \theta} - \frac{1-\sin \theta}{2\sin \theta}\log \frac{1-\sin \theta}{2\sin \theta} + o(1),\]
    where $o(1) \rightarrow 0$ as $r \rightarrow \infty$ and $o(1)$ is explicit for $\theta$.
    
    In particular, 
    \[A(r,\theta) \ll \left[\exp\left(\frac{1+\sin \theta}{2\sin \theta}\log\frac{1+\sin \theta}{2\sin \theta} - \frac{1-\sin \theta}{2\sin \theta}\log \frac{1-\sin \theta}{2\sin \theta}+0.001\right)\right]^r.\]
\end{theorem}
\begin{proof}
    See \cite{KL78}.
\end{proof}

When $\theta>\pi/2$, then $A(r,\theta)$ is bounded independently of $r$.

\begin{theorem}\label{spherical_code_bound2}
    For fixed $\theta>\pi/2$, 
    \[A(r,\theta) \ll 1.\]
\end{theorem}
\begin{proof}
    Let $X = \{x_1,\ldots,x_n\}$ be a spherical code with respect to $r$ and $\theta$. From  
    \[0 \leq \langle x_1+\cdots+x_n,x_1+\cdots+x_n \rangle \leq n + n(n-1)\cos \theta,\]
    we obtain 
    \[n \leq 1-\frac{1}{\cos \theta}.\]
    Therefore, 
    \[A(r,\theta) \leq 1-\frac{1}{\cos \theta}.\]
\end{proof}

Now we will briefly explain why a bound for $A(r,\theta)$ is required in deriving a bound for the integral points.

Suppose $E_D(\bbq)$ has rank $r$. Then $E_D(\bbq) \otimes_\bbz \bbr$ is isomorphic to $\bbr^r$ and the canonical height $\hat{h}$ on $E_D(\bbq)$ extends $\bbr$-linearly to a positive definite quadratic form on $E_D(\bbq) \otimes_\bbz \bbr \cong \bbr^r$. Therefore, we have an associated inner product $\langle \:\cdot\:.\:\cdot\: \rangle$ on $E_D(\bbq) \otimes_\bbz \bbr \cong \bbr^r$. Let $P,Q \in E_D(\bbq)$ be non-torsion points. The angle $\theta_{P,Q}$ between $P,Q$ is defined by the formula 
\[\cos \theta_{P,Q} := \frac{\langle P,Q \rangle}{2\lVert P \rVert\lVert Q \rVert} = \frac{\hat{h}(P+Q) - \hat{h}(P) - \hat{h}(Q)}{2\sqrt{\hat{h}(P)\hat{h}(Q)}} = \frac{\hat{h}(P) +\hat{h}(Q) - \hat{h}(P-Q)}{2\sqrt{\hat{h}(P)\hat{h}(Q)}}.\]

Suppose a finite set $X$ of non-torsion points in $E_D(\bbq)$ satisfies 
\begin{equation}\label{repeling_property}
    \cos \theta_{P,Q} \leq \cos \theta_0,\quad P,Q \in X
\end{equation}
for some $\theta_0>0$. Then the image of $X$ under 
\[X \longrightarrow E_D(\bbq) \otimes \bbr,\quad P \longmapsto P \otimes \frac{1}{\sqrt{\hat{h}(P)}}\]
forms a spherical code with respect to $r$ and $\theta_0$. Therefore, we have $\lvert X \rvert \leq A(r,\theta_0)$.

As we will see later, if we choose integral points carefully, then we can make them satisfy the gap principle \eqref{repeling_property}. Then we can use Theorem~\ref{spherical_code_bound1} and Theorem~\ref{spherical_code_bound2} to bound those integral points.

\section{Bound for small points}

In this section, we bound the set 
\[E_D(\bbz)_{small} = \{P \in E_D(\bbz)\:|\:\hat{h}(P) \leq 1.5\log D\}.\]

\begin{proposition}\label{S_1_bound}
    For sufficiently large $D$, 
    \[\lvert E_D(\bbz)_{small} \rvert \ll 4^r.\]
\end{proposition}
\begin{proof}
    Divide $E_D(\bbz)_{small}$ into cosets of $4E_D(\bbq)$. For any point $R \in E_D(\bbq)$, define 
    \[\mathcal{S}(R) := \{P \in E_D(\bbz)_{small}\:|\:P-R \in 4E_D(\bbq)\}.\]
    We will prove 
    \begin{equation}\label{S_R_estimate}
        \lvert \mathcal{S}(R) \rvert \ll 1,\quad R \in E_D(\bbq).
    \end{equation}
    By Lemma~\ref{torsion}, there are at most $4^{r+1}$ cosets of $4E_D(\bbq)$ for sufficiently large $D$, so \eqref{S_R_estimate} proves the proposition.
    
    Let $\{P_1,\ldots,P_n\} \subseteq \mathcal{S}(R)$ be maximal with the property that $P_i-P_j$ is non-torsion whenever $i \neq j$. For $i \neq j$, write $P_i-P_j=4S$ for some non-torsion point $S$. Then by Lemma~\ref{canonical_height_lower_bound}, 
    \[\hat{h}(P_i-P_j) = 16\hat{h}(S) \geq 4\log D + 16c_3 \geq 3.5\log D\]
    for sufficiently large $D$. However, 
    \[\hat{h}(P_i) \leq 1.5\log D,\quad \hat{h}(P_j) \leq 1.5\log D.\]
    Therefore, 
    \begin{equation}\label{S_R_angle}
        \cos \theta_{P_i,P_j} = \frac{\hat{h}(P_i) + \hat{h}(P_j)-\hat{h}(P_i-P_j)}{2\sqrt{\hat{h}(P_i)\hat{h}(P_j)}} \leq -\frac{0.5\log D}{3\log D} = -\frac{1}{6} < 0.
    \end{equation}
    \eqref{S_R_angle} shows that the angle between any two distinct points of $\{P_1,\ldots,P_n\}$ is bounded from below by $\theta_0>\pi/2$. By Theorem~\ref{spherical_code_bound2}, $n$ is bounded by an absolute constant. By maximality, every $P \in \mathcal{S}(R)$ differs from some $P_i$ by a torsion point. By Lemma~\ref{torsion}, $\lvert E_D(\bbq)_{tors} \rvert \leq 4$ for sufficiently large $D$, so we conclude that 
    \[\lvert \mathcal{S}(R) \rvert \leq 4n \ll 1.\]
\end{proof}

\begin{rmk}
    Note that we did not use the fact that $E_D(\bbz)_{small}$ consists of integral points. Therefore, the proof is valid for the set 
    \[\{P \in E_D(\bbq)\:|\:\hat{h}(P) \leq 1.5\log D\}.\]
\end{rmk}

\section{Bound for medium-small points}

In this section, we bound the set 
\[E_D(\bbz)_{medium-small} = \{P \in E_D(\bbz)\:|\:1.5 \log D \leq \hat{h}(P) \leq 20\log D\}.\]

We need two lemmas to prove Proposition~\ref{MS_bound}. The first lemma says that for sufficiently large $D$, integral points with $x(P) \leq MD$ must lie in $E_D(\bbz)_{small}$.

\begin{lemma}\label{small_x(P)}
    Let $P \in E_D(\bbz)$ and $x(P) \leq MD$. Then for sufficiently large $D$, $\hat{h}(P) < 1.5\log D$.
\end{lemma}
\begin{proof}
    Let $f(x) = x^3+D^2Ax+D^3B$ for $x \in \bbr$. By the choice of $M$, $f'(x)>0$ on $x \in (-\infty,-MD)$ and $f(-MD)<0$. Therefore, $x(P) \geq -MD$. 

    By using the isomorphism $\phi_D : E_D \rightarrow E^D$, we have  
    \[\hat{h}(P) = \hat{h}(\phi_D(P)) \leq h(\phi_D(P)) + c_2 = h(x(P)/D) + c_2.\]
    Since $-MD \leq x(P) \leq MD$ and $x(P)$ is an integer, 
    \[h(x(P)/D) \leq \log D + \log M.\]
    Therefore, 
    \[\hat{h}(P) \leq \log D + \log M + c_2 < 1.5\log D\]
    for sufficiently large $D$.
\end{proof}

The second lemma is a technical lemma needed in optimization process.

\begin{lemma}\label{f(a,b)_maximum}
    Let $0<\alpha<\beta$ and $0 < c<\alpha$ be fixed constants. Define 
    \[f(a,b) = \frac{a^2+b^2-c^2}{2ab},\quad a,b \in [\alpha, \beta].\]
    Then 
    \[\max_{a,b \in [\alpha, \beta]} f(a,b) = \max \left\{\frac{\alpha^2+\beta^2-c^2}{2\alpha\beta},1-\frac{c^2}{2\beta^2}\right\}.\]
\end{lemma}
\begin{proof}
    We have 
    \[\frac{\partial f}{\partial a} = \frac{a^2-b^2+c^2}{2a^2b},\quad \frac{\partial f}{\partial b} = \frac{-a^2+b^2+c^2}{2ab^2}.\]
    Fix $b_0 \in [\alpha,\beta]$. On the line segment $[\alpha,\beta] \times \{b_0\}$, we have 
    \[\frac{\partial f}{\partial a} = \frac{a^2-b_0^2+c^2}{2a^2b_0}.\]
    If $\sqrt{b_0^2-c^2} \leq \alpha$, then $f(a,b_0)$ increases on $[\alpha,\beta]$. If $\alpha < \sqrt{b_0^2-c^2} < \beta$, then $f(a,b_0)$ decreases on $[\alpha,\sqrt{b_0^2-c^2}]$ and increases on $[\sqrt{b_0^2-c^2},\beta]$. In any case, the maximum of $f(a,b_0)$ cannot happen in the interior of $[\alpha,\beta]$. By symmetry, for any $a_0 \in [\alpha,\beta]$, the maximum of $f(a_0,b)$ cannot happen in the interior of $[\alpha,\beta]$. Therefore, the maximum of $f(a,b)$ can happen only at 
    \[(a,b) \in \{(\alpha,\alpha),(\alpha,\beta),(\beta,\alpha),(\beta,\beta)\}.\]

    It is clear that $f(\alpha,\alpha) < f(\beta,\beta)$. Thus 
    \[\max_{\alpha \leq a,b \leq \beta} f(a,b) = \max \left\{\frac{\alpha^2+\beta^2-c^2}{2\alpha\beta},1-\frac{c^2}{2\beta^2}\right\}.\]
\end{proof}

We now estimate $E_D(\bbz)_{medium-small}$.

\begin{proposition}\label{MS_bound}
    For sufficiently large $D$, 
    \[\lvert E_D(\bbz)_{medium-small} \rvert \ll 4^r.\]
\end{proposition}
\begin{proof}
    For $2 \leq n \leq 20$, define 
    \[\mathcal{MS}_n := \{P \in E_D(\bbz)\:|\: (n-0.5)\log D \leq \hat{h}(P) \leq (n+0.5)\log D\}\]
    and 
    \[\mathcal{MS}_{n}^+ := \{P \in \mathcal{MS}_{n}\:|\:y(P)>0\}.\]
    Then 
    \[\lvert E_D(\bbz)_{medium-small} \rvert \leq \sum_{n=2}^{20} \lvert \mathcal{MS}_n \rvert = 2\sum_{n=2}^{20} \lvert \mathcal{MS}_n^+ \rvert.\]
    Therefore, it suffices to prove 
    \begin{equation}\label{MS_n_bound}
        \lvert \mathcal{MS}_n^+ \rvert \ll 4^r,\quad 2 \leq n \leq 20.
    \end{equation}
    
    Fix $2 \leq n \leq 20$. Let $P,Q \in \mathcal{MS}_n^+$ be distinct points. First, by Lemma~\ref{small_x(P)}, $x(P),x(Q) \geq MD$ for sufficiently large $D$. Then by Lemma~\ref{P+Q_negative_bound}, 
    \[x(P-Q) \geq \min\{x(P),x(Q)\},\]
    so 
    \[h(P-Q) \geq \min\{h(P),h(Q)\}.\]
    By Lemma~\ref{Weil_canonical_height}, 
    \[h(P) \geq \hat{h}(P)-c_2 \geq (n-0.55)\log D,\quad h(Q) \geq \hat{h}(Q)-c_2 \geq (n-0.55)\log D,\]
    thus 
    \[h(P-Q) \geq (n-0.55)\log D.\]
    for sufficiently large $D$. Hence by Lemma~\ref{Weil_canonical_height}, 
    \[\hat{h}(P-Q) \geq h(P-Q)+c_1-\log D \geq (n-1.55)\log D + c_1 \geq (n-1.6)\log D\]
    for sufficiently large $D$. Then 
    \[\cos \theta_{P,Q} = \frac{\hat{h}(P) + \hat{h}(Q)-\hat{h}(P-Q)}{2\sqrt{\hat{h}(P)\hat{h}(Q)}} \leq \frac{\hat{h}(P) + \hat{h}(Q)-(n-1.6)\log D}{2\sqrt{\hat{h}(P)\hat{h}(Q)}}.\]
    By Lemma~\ref{f(a,b)_maximum}, 
    \begin{equation}\label{MS_n_angle}
        \cos\theta_{P,Q} \leq \max\left\{\frac{n+1.6}{2\sqrt{n^2-0.25}},1-\frac{n-1.6}{2(n+0.5)}\right\}.
    \end{equation}
    \eqref{MS_n_angle} shows that the angle between any two distinct points of $\mathcal{MS}_n^+$ is bounded from below. Therefore, Theorem~\ref{spherical_code_bound1} with some computation implies \eqref{MS_n_bound}. The precise values can be found in the appendix.
\end{proof}

\section{Bound for medium-large points}

In this section, we bound the set 
\[E_D(\bbz)_{medium-large}  = \{P \in E_D(\bbz)\:|\:20\log D \leq \hat{h}(P) \leq 2200\log D\}.\]

We first prove that if $\hat{h}(P)$ and $\hat{h}(Q)$ are close together, then the angle $\theta_{P,Q}$ between $P$ and $Q$ is bounded from below. This is the gap principle for $E_D(\bbz)_{medium-large}$.

\begin{lemma}\label{medium_points_bound}
    Let $P,Q \in E_D(\bbz)_{medium-large}$ satisfy $x(P) \neq x(Q)$ and 
    \[\max \left\{\frac{\hat{h}(Q)}{\hat{h}(P)},\frac{\hat{h}(P)}{\hat{h}(Q)}\right\} \leq 1.1.\] 
    Then for sufficiently large $D$, 
    \[\cos \theta_{P,Q} \leq 0.63.\]
\end{lemma}
\begin{proof}
    Without loss of generality, assume $x(P)<x(Q)$. By Lemma~\ref{small_x(P)}, $x(P),x(Q) \geq MD$. By Lemma~\ref{Weil_height_estimate}, 
    \[h(P+Q) \leq h(P) + 2h(Q) + 2.9.\]
    By Lemma~\ref{Weil_canonical_height}, 
    \[\hat{h}(P+Q) \leq \hat{h}(P) + 2\hat{h}(Q) + 4\log D\]
    for sufficiently large $D$. It follows that  
    \[\cos \theta_{P,Q} = \frac{\hat{h}(P+Q) - \hat{h}(P) - \hat{h}(Q)}{2\sqrt{\hat{h}(P)\hat{h}(Q)}} \leq \frac{1}{2}\sqrt{\frac{\hat{h}(Q)}{\hat{h}(P)}} + 0.1 \leq 0.63.\]
\end{proof}

\begin{proposition}\label{ML_bound}
    For sufficiently large $D$, 
    \[\lvert E_D(\bbz)_{medium-large} \rvert \ll 4^r.\]
\end{proposition}
\begin{proof}
    Note that $110 \leq (1.1)^{50}$. For $1 \leq n \leq 50$, define 
    \[\mathcal{ML}_n := \{P \in E_D(\bbz)\:|\:20\cdot(1.1)^{n-1}\log D \leq \hat{h}(P) \leq 20\cdot(1.1)^n\log D\}\]
    and 
    \[\mathcal{ML}_n^+ := \{P \in \mathcal{ML}_n\:|\:y(P)>0\}.\]
    Then 
    \[\lvert E_D(\bbz)_{medium-large} \rvert \leq \sum_{n=1}^{50} \lvert \mathcal{ML}_n \rvert = 2\sum_{n=1}^{50} \lvert \mathcal{ML}_n^+ \rvert.\]
    Therefore, it suffices to prove 
    \begin{equation}\label{ML_n_bound}
        \lvert \mathcal{ML}_n^+ \rvert \ll 4^r,\quad 1 \leq n \leq 50.
    \end{equation}

    Fix $1 \leq n \leq 50$. Let $P,Q \in \mathcal{ML}_n^+$ be distinct points. Then $x(P) \neq x(Q)$. Note that 
    \[\max \left\{\frac{\hat{h}(Q)}{\hat{h}(P)},\frac{\hat{h}(P)}{\hat{h}(Q)}\right\} \leq \frac{20\cdot(1.1)^n\log D}{20\cdot(1.1)^{n-1}\log D} = 1.1.\] 
    Thus by Lemma~\ref{medium_points_bound}, 
    \begin{equation}\label{ML_n_angle}
        \cos \theta_{P,Q} \leq 0.63.
    \end{equation}
    \eqref{ML_n_angle} shows that the angle between any two distinct points of $\mathcal{ML}_n^+$ is bounded from below. Therefore, Theorem~\ref{spherical_code_bound1} with some computation implies \eqref{ML_n_bound}. Precisely, we have 
    \[\lvert \mathcal{ML}_n^+ \rvert \ll (1.55)^r,\quad 1 \leq n \leq 50.\]
\end{proof}

\section{Diophantine approximations}

In this section, we prove Lemma~\ref{Diophantine_approximation}. Roughly speaking, Lemma~\ref{Diophantine_approximation} says that large integral points give rational approximations to an algebraic number with exponent $>2$. By the celebrated Roth's theorem \cite{Rot55}, there are only finitely many such rational approximations, so we are able to bound large integral points. This kind of argument was originally given by Alpoge \cite{Alp14}, and subsequently refined by Chan \cite{Cha22}. The form of Diophantine approximation applied in this work is analogous to \cite{Cha22}[Lemma~6.1]. 

We begin with a simple lemma bounding the derivative of a polynomial.

\begin{lemma}\label{Mahler_bound}
    Let 
    \[f(x) = a_0x^m + a_1x^{m-1} + \cdots + a_m = a_0 \prod_{h=1}^m (x-\alpha_h),\quad m\geq 2\]
    be a polynomial with real or complex coefficients. Let $D(f)$ be the discriminant of $f$ and $L(f) = \lvert a_0 \rvert + \lvert a_1 \rvert + \cdots + \lvert a_m \rvert$. Then 
    \[\lvert f'(\alpha_r) \rvert \geq (m-1)^{-(m-1)/2}\lvert D(f) \rvert^{1/2}L(f)^{-(m-2)}.\]
\end{lemma}
\begin{proof}
    This is in the last line of p.262 of \cite{Mah64}.
\end{proof}

\begin{lemma}\label{Diophantine_approximation}
    Let $P \in E_D(\bbz)$ and let $P=3Q+R$ for $Q,R \in E_D(\bbq)$. Assume that $x(R) \geq MD$, $h(P) > 1000h(R)$, and $h(P) > 2000\log D$. Take $S \in \frac{1}{3}R$ such that $\lvert x(Q)-x(S) \rvert$ is minimum. Then for sufficiently large $D$, 
    \begin{equation}\label{diopahntine_equation_1}
        h(Q) > 5.7h(S)
    \end{equation}
    and 
    \begin{equation}\label{diopahntine_equation_2}
        \frac{\log \lvert x(Q)-x(S) \rvert}{h(Q)} < -2.7.
    \end{equation}
\end{lemma}
\begin{proof}
    Put $\lambda = 1/1000$ and $\delta = 1/2000$ so that $h(R) < \lambda h(P)$ and $\log D < \delta h(P)$.

    We first compare $h(P)$ and $h(Q)$. By the triangle inequality,
    \[3\sqrt{\hat{h}(Q)} - \sqrt{\hat{h}(R)} \leq \sqrt{\hat{h}(P)} \leq 3\sqrt{\hat{h}(Q)} + \sqrt{\hat{h}(R)}.\]
    By squaring, 
    \[9\hat{h}(Q)+\hat{h}(R) - 6\sqrt{\hat{h}(Q)}\sqrt{\hat{h}(R)} \leq \hat{h}(P) \leq 9\hat{h}(Q)+\hat{h}(R) + 6\sqrt{\hat{h}(Q)}\sqrt{\hat{h}(R)}.\]
    Using 
    \[2\sqrt{\hat{h}(Q)}\sqrt{\hat{h}(R)} \leq \hat{h}(Q) + \hat{h}(R),\]
    we obtain 
    \begin{equation}\label{triangle_inequality}
        6\hat{h}(Q)-2\hat{h}(R) \leq \hat{h}(P) \leq 12\hat{h}(Q)+4\hat{h}(R).
    \end{equation}
    By Lemma~\ref{Weil_canonical_height}, 
    \[h(P) - \log D + c_1 \leq \hat{h}(P) \leq h(P) + \log D + c_2.\]
    Since $\log D \leq \delta h(P)$, we have 
    \begin{equation}\label{P_estimate}
        (1-\delta)h(P)+c_1 \leq \hat{h}(P) \leq (1+\delta)h(P) + c_2.
    \end{equation}
    The same estimate shows 
    \begin{equation}\label{Q_estimate}
        h(Q) -\delta h(P) +c_1 \leq \hat{h}(Q) \leq h(Q) + \delta h(P)+c_2.
    \end{equation}
    For $R$, using $h(R) \leq \lambda h(P)$ shows 
    \begin{equation}\label{R-estimate}
        \hat{h}(R) \leq (\lambda+\delta)h(P) + c_2.
    \end{equation}
    Putting \eqref{P_estimate}, \eqref{Q_estimate}, and \eqref{R-estimate} into \eqref{triangle_inequality} and then clearing, we have 
    \[6h(Q) \leq (1+2\lambda+9\delta)h(P) + c_5\] 
    and 
    \[(1-4\lambda-17\delta)h(P)+c_6 \leq 12h(Q).\]
    Dividing by $h(P)$ gives 
    \begin{equation}\label{h(Q)/h(P)_estimate}
        \frac{1-4\lambda-17\delta}{12} + \frac{c_6}{h(P)} \leq \frac{h(Q)}{h(P)} \leq \frac{1+2\lambda+9\delta}{6} + \frac{c_5}{h(P)}.
    \end{equation}
    
    By Lemma~\ref{Weil_canonical_height}, 
    \[h(S) \leq \hat{h}(S) + \log D - c_1.\]
    Using $9\hat{h}(S) = \hat{h}(R)$, \eqref{R-estimate}, and $\log D \leq \delta h(P)$, we have 
    \[h(S) \leq \frac{10\lambda+\delta}{9}h(P) + c_7 \leq \frac{10\lambda+\delta+0.1}{9}h(P)\]
    for sufficiently large $D$. Thus 
    \begin{equation}\label{PS-estimate}
        h(P) \geq \frac{9}{10\lambda+\delta+0.1}h(S).
    \end{equation}
    The left inequality of \eqref{h(Q)/h(P)_estimate} gives 
    \begin{equation}\label{h(Q)_lower_bound}
        h(Q) \geq \frac{0.9-4\lambda-17\delta}{12}h(P).
    \end{equation}
    for sufficiently large $D$. Combining \eqref{PS-estimate} and \eqref{h(Q)_lower_bound} gives 
    \[h(Q) \geq \frac{3(0.9-4\lambda-17\delta)}{4(10\lambda+\delta+0.1)}h(S) > 5.7h(S).\]

    We now begin the proof of \eqref{diopahntine_equation_2}. Note that 
    \[x(R) = x(3T) = \frac{\phi_3(T)}{\psi_3(T)^2},\quad 3T = R\]
    where $\psi_n(x)$ is an $n$-th division polynomial and $\phi_n = x\psi_n(x)^2-\psi_{n+1}(s)\psi_{n-1}(x)$. Define 
    \[f_R(X) := \prod_{3T=R} (X-x(T)) = \phi_3(X) - x(R)\psi_3(X)^2\]
    as a polynomial in $X$. Put $X=x(Q)$ to get 
    \[\prod_{3T=R} (x(Q)-x(T)) = \phi_3(Q) - x(R)\psi_3(Q)^2 = \psi_3(Q)^2(x(3Q)-x(R)).\]
    Note that 
    \begin{align*}
        x(P) &= x(3Q+R) \\
        &= \frac{(x(3Q)x(R)+D^2A)(x(3Q)+x(R))+2D^3B - 2y(3Q)y(R)}{(x(3Q)-x(R))^2}.
    \end{align*}
    Thus 
    \begin{align}\label{division_polynomial_equation}
    \begin{split}
        &x(P)\left(\prod_{3T=R} (x(Q)-x(T))\right)^2  \\
        &= \psi_3(Q)^4((x(3Q)x(R)+D^2A)(x(3Q)+x(R))+2D^3B - 2y(3Q)y(R)).
    \end{split}
    \end{align}
    We will estimate the right hand side. 
    
    Note that $x(R) \leq x(P)^\lambda \leq x(P)$. In particular, 
    \begin{equation}\label{x(P)/x(R)_estimate}
        \frac{x(P)}{x(R)} \geq x(P)^{1-\lambda} \gg 1.
    \end{equation}
    By Lemma~\ref{3P_bound}, Lemma~\ref{P+Q_positive_bound}, Lemma~\ref{P+Q_negative_bound}, \eqref{x(P)/x(R)_estimate}, and the relation $P-R=3Q$, we have 
    \begin{equation}\label{x(Q)_bound}
        x(Q) \ll x(R).
    \end{equation}
    By using \eqref{x(Q)_bound}, Lemma~\ref{P+Q_positive_bound}, and Lemma~\ref{P+Q_negative_bound}, we have 
    \[\psi_3(Q)^4 \ll x(R)^{16}\]
    and
    \[(x(3Q)x(R)+D^2A)(x(3Q)+x(R))+2D^3B - 2y(3Q)y(R) \ll x(R)^3\]
    and thus the right hand side of \eqref{division_polynomial_equation} is bounded by 
    \[\ll x(R)^{19} \leq x(P)^{19\lambda}.\]
    Dividing by $x(P)$ gives 
    \[\left(\prod_{3T=R} (x(Q)-x(T))\right)^2 \ll x(P)^{-1+19\lambda}.\]
    By taking logs and dividing by $h(P)$ gives 
    \begin{equation}\label{log_1}
        \frac{\log \prod_{3T=R} \lvert x(Q)-x(T) \rvert}{h(P)} \leq -\frac{1}{2} + \frac{19\lambda}{2} + \frac{c_8}{h(P)}.
    \end{equation}

    Now let $\alpha = x(S)$. Then 
    \[f_R'(\alpha) = \prod_{\substack{3T=R \\ T \neq S}} (\alpha - x(T)).\]
    By Lemma~\ref{Mahler_bound}, 
    \[\lvert f_R'(\alpha) \rvert \geq 8^{-4}\lvert D(f_R) \rvert^{1/2}L(f_R)^{-7}.\]
    Let $x(R) = r/s$ where $\gcd(r,s)=1$. Then $h(R) = \log r$ and $r \geq MDs$. From $sf_R(X) \in \bbz(X)$, $\lvert D(sf_R) \rvert \geq 1$, so $\lvert D(f_R) \rvert \geq s^{-16}$. Also we can easily obtain 
    \[L(f_R) \ll (MD)^8x(R).\]
    Thus 
    \[\lvert f_R'(\alpha) \rvert \gg r^{-8}(MD)^{-55}.\]
    By the triangle inequality, 
    \[\lvert x(S) - x(T) \rvert \leq \lvert x(Q) - x(S) \rvert+\lvert x(Q) - x(T) \rvert  \leq 2\lvert x(Q) - x(T) \rvert\]
    for all $T \in \frac{1}{3}R$ such that $T \neq S$. Thus 
    \[\prod_{\substack{3T=R \\ T \neq S}} \lvert \alpha - x(T) \rvert \leq 2^8\prod_{\substack{3T=R \\ T \neq S}} \lvert x(Q) - x(T) \rvert.\]
    Hence, 
    \[\prod_{\substack{3T=R \\ T \neq S}} \lvert x(Q) - x(T) \rvert \gg \prod_{\substack{3T=R \\ T \neq S}} \lvert \alpha - x(T) \rvert = \lvert f_R'(\alpha) \rvert \gg r^{-8}(MD)^{-55}.\]
    By taking logs, we have 
    \[\log \prod_{\substack{3T=R \\ T \neq S}} \lvert x(Q) - x(T) \rvert \geq - 8h(R) - 55\log D + c_9.\]
    Dividing by $h(P)$, we have 
    \begin{equation}\label{log_2}
        \frac{\log \prod_{\substack{3T=R \\ T \neq S}} \lvert x(Q) - x(T) \rvert }{h(P)} \geq - 8\lambda - 55\delta + \frac{c_9}{h(P)}.
    \end{equation}

    Combining \eqref{log_1} and \eqref{log_2} gives 
    \[\frac{\log \lvert x(Q)-x(S) \rvert}{h(P)} \leq -\frac{1}{2} + \frac{27}{2}\lambda+55\delta + \frac{c_{10}}{h(P)}.\]
    Using $h(P)>2000\log D$ gives 
    \begin{equation}\label{log_3}
        \frac{\log \lvert x(Q)-x(S) \rvert}{h(P)} \leq -\frac{1}{2} + \frac{27}{2}\lambda+55\delta + \frac{c_{11}}{\log D}.
    \end{equation}

    From \eqref{h(Q)/h(P)_estimate}, we have 
    \[\frac{h(Q)}{h(P)} \leq \frac{1+2\lambda+9\delta}{6} + \frac{c_{12}}{\log D}.\]
    Using the identity 
    \[\frac{1}{a+b} = \frac{1}{a} - \frac{b}{a(a+b)} \geq \frac{1}{a} - \frac{b}{2a^2},\quad 0 < b \leq a\]
    we have 
    \begin{equation}\label{h(P)/h(Q)_lower_bound}
        \frac{h(P)}{h(Q)} \geq \frac{6}{1+2\lambda+9\delta} + \frac{c_{13}}{\log D}.
    \end{equation}

    Combining \eqref{log_3} and \eqref{h(P)/h(Q)_lower_bound} gives 
    \begin{align*}
        \frac{\log \lvert x(Q)-x(S) \rvert}{h(Q)} &\leq \left(-\frac{1}{2} + \frac{27}{2}\lambda+55\delta + \frac{c_{12}}{\log D}\right)\left(\frac{6}{1+2\lambda+9\delta} + \frac{c_{13}}{\log D}\right) \\
        &\leq -\frac{3(1-27\lambda-110\delta)}{1+2\lambda+9\delta} + \frac{c_{14}}{\log D} \\
        &< -2.7
    \end{align*}
    for sufficiently large $D$.
\end{proof}

\section{Quantitative Roth's theorem}

As mentioned above, Lemma~\ref{Diophantine_approximation} says that if $P=3Q+R$ is a very large integral point, then $x(Q)$ is a rational approximation to $x(S)$ with exponent $>2.7$. By Roth's theorem, it is known that the number of such $x(Q)$ is finite. In order to bound the number, we need a quantitative version of Roth's theorem. 

The first quantitative Roth's theorem was given by Davenport and Roth \cite{DR55} in 1955. Their bound was subsequently improved by Mignotte \cite{Mig74}, Bombieri and van der Poorten \cite{BP88}, Silverman \cite{Sil87}, Gross \cite{Gro90}, and Evertse \cite{Eve97}. We state a quantitative Roth's theorem given by Evertse \cite{Eve10}.

\begin{theorem}\label{Quantitative_Roth_theorem}
    Fix an embedding $\overline{\bbq} \hookrightarrow \bbc$, so that we have an infinite place $\lvert \:\cdot \:\rvert$ on $\overline{\bbq}$. Let $\alpha \in \overline{\bbq}$ be an algebraic number of degree $d$ and $\epsilon>0$. Then the number of $\beta \in \bbq$ satisfying the simultaneous inequalities 
    \[h(\beta) \geq \max\{h(\alpha),\log 2\},\quad \frac{\log \lvert \alpha-\beta \rvert}{h(\beta)} < -(2+\epsilon)\]
    is at most 
    \[2^{25}\epsilon^{-3}\log (2d)\log(\epsilon^{-1}\log (2d)).\]
\end{theorem}
\begin{proof}
    This is \cite[Theorem~1.1]{Eve10}.
\end{proof}

\section{Bound for large points}

In this section, we bound the set 
\[E_D(\bbz)_{large} = \{P \in E_D(\bbz)\:|\:\hat{h}(P) \geq 2200\log D\}.\]

We first divide $E_D(\bbz)_{large}$ into cosets of $3E_D(\bbq)$. For any point $R \in E_D(\bbq)$, define 
\[\mathcal{L}(R) := \{P \in E_D(\bbz)_{large}\:|\:P-R \in 3E_D(\bbq)\}.\]
For each coset $R+3E_D(\bbq)$, if $\mathcal{L}(R)$ is non-empty, then we choose $R$ to be the point with minimum canonical height among points with $x$-coordinate $\geq MD$. Then define 
\[\mathcal{L}(R)^* := \{P \in \mathcal{L}(R)\:|\:\hat{h}(P) \leq 1050\hat{h}(R)\}\]
and 
\[\mathcal{L}(R)^{**} := \{P \in \mathcal{L}(R)\:|\:\hat{h}(P) > 1050\hat{h}(R)\}.\]

We first estimate $\mathcal{L}(R)^*$. The method is same as that of $E_D(\bbz)_{medium-large}$, but we set the bound sharper.

\begin{lemma}\label{large_points_bound}
    Let $P,Q \in E_D(\bbz)_{large}$ satisfy $x(P) \neq x(Q)$ and 
    \[\max \left\{\frac{\hat{h}(Q)}{\hat{h}(P)},\frac{\hat{h}(P)}{\hat{h}(Q)}\right\} \leq 1.01.\] 
    Then for sufficiently large $D$, 
    \[\cos \theta_{P,Q} \leq 0.504.\]
\end{lemma}
\begin{proof}
    Without loss of generality, assume $x(P)<x(Q)$. The same argument as in Lemma~\ref{medium_points_bound} gives 
    \[\hat{h}(P+Q) - \hat{h}(P) - \hat{h}(Q) \leq \hat{h}(Q) + 4\log D\]
    for sufficiently large $D$. It follows that 
    \[\cos \theta_{P,Q} = \frac{\hat{h}(P+Q) - \hat{h}(P) - \hat{h}(Q)}{2\sqrt{\hat{h}(P)\hat{h}(Q)}} \leq \frac{1}{2}\sqrt{\frac{\hat{h}(Q)}{\hat{h}(P)}} + 0.001 \leq 0.504.\]
\end{proof}

\begin{proposition}\label{LR*_bound}
    For sufficiently large $D$, 
    \[\lvert \mathcal{L}(R)^* \rvert \ll (1.33)^r.\]
\end{proposition}
\begin{proof}
    Note that $1050 \leq (1.01)^{700}$. For $1 \leq n \leq 700$, define 
    \[\mathcal{L}(R)^*_n := \{P \in \mathcal{L}(R)\:|\:(1.01)^{n-1} \hat{h}(R)\leq \hat{h}(P) \leq (1.01)^n \hat{h}(R)\}\]
    and 
    \[(\mathcal{L}(R)^*_n)^+ := \{P \in \mathcal{L}(R)_{1,n}\:|\:y(P)>0\}.\]
    Then 
    \[\lvert \mathcal{L}(R)^* \rvert \leq \sum_{n=1}^{700} \lvert \mathcal{L}(R)^*_n \rvert = 2\sum_{n=1}^{700} \lvert (\mathcal{L}(R)^*_n)^+ \rvert.\]
    Therefore, it suffices to prove 
    \begin{equation}\label{LR_1n_bound}
        \lvert (\mathcal{L}(R)^*_n)^+ \rvert \ll (1.33)^r,\quad 1 \leq n \leq 700.
    \end{equation}

    Fix $1 \leq n \leq 700$. Let $P,Q \in (\mathcal{L}(R)^*_n)^+$ be distinct points. Then $x(P) \neq x(Q)$. Note that 
    \[\max \left\{\frac{\hat{h}(Q)}{\hat{h}(P)},\frac{\hat{h}(P)}{\hat{h}(Q)}\right\} \leq 1.01.\] 
    Thus by Lemma~\ref{large_points_bound}, 
    \begin{equation}\label{LR_1n_angle}
        \cos \theta_{P,Q} \leq 0.504.
    \end{equation}
    \eqref{LR_1n_angle} shows that the angle between any two distinct points of $(\mathcal{L}(R)^*_n)^+$ is bounded from below. Therefore, Theorem~\ref{spherical_code_bound1} with some computation implies \eqref{LR_1n_bound}.
\end{proof}

We next estimate $\mathcal{L}(R)^{**}$. For each $Q \in E_D(\bbq)$, let $S_Q \in \frac{1}{3}R$ be the point such that $\lvert x(Q) - x(S_Q) \rvert$ is minimum. Define 
\[\mathcal{L}(R;S)^{**} = \{P \in \mathcal{L}(R)^{**}\:|\:S_Q=S\text{ for some }Q \in \frac{1}{3}(P-R)\}.\]
Clearly, 
\begin{equation}\label{LR_decomposition}
    \mathcal{L}(R)^{**} = \bigcup_{3S=R} \mathcal{L}(R;S)^{**}.
\end{equation}

\begin{proposition}\label{LR**_bound}
    For sufficiently large $D$, 
    \[\lvert \mathcal{L}(R)^{**} \rvert \ll 1.\]
\end{proposition}
\begin{proof}
    By \eqref{LR_decomposition}, it suffices to prove 
    \begin{equation}\label{LRS_decomposition}
        \lvert \mathcal{L}(R;S)^{**} \rvert \ll 1,\quad S \in \frac{1}{3}R.
    \end{equation}
    Let $P \in \mathcal{L}(R;S)^{**}$ and write $P=3Q+R$ where $S_Q=S$. By Lemma~\ref{Weil_canonical_height}, 
    \[h(P) > 2000\log D,\quad h(P) > 1000h(R).\]
    Therefore, by Lemma~\ref{Diophantine_approximation}, 
    \[h(Q) \geq \max\{h(S),\log 2\},\quad \frac{\log \lvert x(S) - x(Q) \rvert}{h(Q)} < -2.7.\]
    Note that $x(S)$ is an algebraic number of degree at most $9$. By Theorem~\ref{Quantitative_Roth_theorem} with $d \leq 9$ and $\epsilon = 0.7$, \eqref{LRS_decomposition} is proved.
\end{proof}

\begin{proposition}\label{large_bound}
    For sufficiently large $D$, 
    \[\lvert E_D(\bbz)_{large} \rvert \ll 4^r.\]
\end{proposition}
\begin{proof}
    By Proposition~\ref{LR*_bound} and Proposition~\ref{LR**_bound}, $\lvert \mathcal{L}(R) \rvert \ll (1.33)^r$ whenever $\mathcal{L}(R)$ is non-empty. By Lemma~\ref{torsion}, there are $3^r$ cosets of $3E_D(\bbq)$ for sufficiently large $D$. Therefore, 
    \[\lvert E_D(\bbz)_{large} \rvert \ll (3.99)^r \leq 4^r.\]
\end{proof}

\appendix

\section{Calculations for Lemma~\ref{P+Q_positive_bound} and Lemma~\ref{P+Q_negative_bound}}

\begin{lemma}\label{f(x)_lower_bound}
    Let $x \geq 1$ be a variable and $\lvert a \rvert,\lvert b \rvert \leq 0.01$ be constants. Define 
    \[f(x) = \frac{(x+a)(x+1)+2b-2\sqrt{x^3+ax+b}\sqrt{1+a+b}}{(1-x)^2}.\]
    Then $f(x) \geq 0.19$ for all $x \geq 1$.
\end{lemma}
\begin{proof}
    Note that 
    \begin{align*}
        f(x) &= \frac{x^2+x+ax+a+2b-2\sqrt{x^3+ax+b}\sqrt{1+a+b}}{(1-x)^2} \\
        &= \frac{x^3+1+ax+a+2b-2\sqrt{x^3+ax+b}\sqrt{1+a+b}}{(1-x)^2} - (x+1) \\
        &= \left(\frac{\sqrt{x^3+ax+b}-\sqrt{1+a+b}}{x-1}\right)^2 - (x+1) \\
        &= \left(\frac{x^2+x+1+a}{\sqrt{x^3+ax+b}+\sqrt{1+a+b}}\right)^2 - (x+1).
    \end{align*}
    From the inequality 
    \[\sqrt{u+v} \leq \sqrt{u} + \frac{v}{2\sqrt{u}},\quad u,v \geq 0,\]
    we have 
    \[\sqrt{x^3+ax+b} \leq \sqrt{x^3+0.01x+0.01} \leq x^{3/2} + \frac{0.01x+0.01}{2x^{3/2}} \leq x^{3/2} + 0.01.\]
    Also 
    \[\sqrt{1+a+b} \leq \sqrt{1.02} \leq 1.01.\]
    Thus 
    \[(\sqrt{x^3+ax+b}+\sqrt{1+a+b})^2 \leq (x^{3/2} + 1.02)^2.\]
    It follows that 
    \[f(x) \geq \left(\frac{x^2+x+0.99}{x^{3/2} + 1.02}\right)^2-(x+1).\]
    
    Define 
    \[h(x) = \frac{x^2+x+0.99}{x^{3/2} + 1.02},\quad g(x) = h(x)^2-(x+1).\]
    We have 
    \begin{align*}
        h'(x) &= \frac{(2x+1)(x^{3/2}+1.02) - (x^2+x+0.99)(1.5x^{1/2})}{(x^{3/2} + 1.02)^2} \\
        &= \frac{0.5x^{5/2}-0.5x^{3/2}+2.04x-1.485x^{1/2}+1.02}{(x^{3/2} + 1.02)^2}.
    \end{align*}
    Thus 
    \begin{align*}
        g'(x) &= 2h(x)h'(x)-1 \\
        &= 2\frac{x^2+x+0.99}{x^{3/2} + 1.02}\frac{0.5x^{5/2}-0.5x^{3/2}+2.04x-1.485x^{1/2}+1.02}{(x^{3/2} + 1.02)^2} - 1 \\
        &= \frac{(x^2+x+0.99)(x^{5/2}-x^{3/2}+4.08x-2.97x^{1/2}+2.04)}{(x^{3/2} + 1.02)^3} - 1 \\
        &= \frac{1.02x^3-2.98x^{5/2}+6.12x^2-7.0812x^{3/2}+6.0792x-2.9403x^{1/2}+0.958392}{(x^{3/2} + 1.02)^3}.
    \end{align*}
    
    We will prove 
    \[1.02t^6-2.98t^5+6.12t^4-7.0812t^3+6.0792t^2-2.9403t+0.958392 > 0,\quad t \geq 1.\]
    Define 
    \[G(t) = 1.02t^6-2.98t^5+6.12t^4-7.0812t^3+6.0792t^2-2.9403t+0.958392,\quad t \geq 1.\]
    One can easily check 
    \begin{gather*}
        G(1) = 1.176092, \:G'(1) = 3.6745,\:G''(1) = 14.1112,\:G^{(3)}(1) = 47.9928, \\
        G^{(4)}(1) = 156.48,\: G^{(5)}(1) = 376.8,\:G^{(6)}(t) = 734.4.
    \end{gather*}
    From $G^{(6)}(t) = 734.4>0$, $G^{(5)}(t)$ is increasing and thus $G^{(5)}(t) \geq 376.8$. This implies that $G^{(4)}(t)$ is increasing and thus $G^{(4)}(t) \geq 156.48$. Repeating this gives $G(t) \geq 1.176092 > 0$.

    Now we have shown that $g'(x)>0$ for all $x \geq 1$. Hence, $g(x)$ is increasing. Since 
    \[g(1) = \left(\frac{2.99}{2.02}\right)^2 - 2 \geq 0.19,\] 
    the proof is over.
\end{proof}

\begin{lemma}\label{f(x)_upper_bound}
    Let $x \geq 1$ be a variable and $\lvert a \rvert,\lvert b \rvert \leq 0.01$ be constants. Define 
    \[f(x) = \frac{(x+a)(x+1)+2b-2\sqrt{x^3+ax+b}\sqrt{1+a+b}}{(1-x)^2}.\]
    Then $f(x) \leq 2$ for all $x \geq 1$.
\end{lemma}
\begin{proof}
    We start from 
    \[f(x) = \left(\frac{x^2+x+1+a}{\sqrt{x^3+ax+b}+\sqrt{1+a+b}}\right)^2 - (x+1).\]
    The inequality $f(x) \leq 2$ is equivalent to 
    \[(x^2+x+1+a)^2 \leq (x+3)(\sqrt{x^3+ax+b}+\sqrt{1+a+b})^2.\]
    Note that 
    \[(\sqrt{x^3+ax+b}+\sqrt{1+a+b})^2 = x^3+ax+(1+a+2b) + 2\sqrt{x^3+ax+b}\sqrt{1+a+b}\]
    and 
    \begin{align*}
        2\sqrt{x^3+ax+b}\sqrt{1+a+b} &\geq 2\sqrt{x^3-0.01x-0.01}\sqrt{0.98} \geq 1.97\sqrt{x^2-x+0.25} \\
        &\geq 1.94(x-0.5) \geq 1.94x - 0.97.
    \end{align*}
    Thus 
    \[(\sqrt{x^3+ax+b}+\sqrt{1+a+b})^2 \geq x^3+(1.94+a)x+(1+a+2b-0.97) \geq x^3+1.93x.\]
    Therefore, it suffices to prove 
    \[(x^2+x+1.01)^2 \leq (x+3)(x^3+1.93x).\]
    Rearranging gives 
    \[x^3-1.09x^2+3.77x-1.0201 \geq 0.\]

    Define 
    \[g(x) = x^3-1.09x^2+3.77x-1.0201.\]
    Then 
    \[g'(x) = 3x^2-2.18x+3.77.\]
    Since the discriminant of $g'(x)$ is negative, $g'(x)>0$ for all $x \in \bbr$. Thus $g(x)$ is increasing for all $x \in \bbr$. Since $g(1) \geq 2.65$, $g(x) \geq g(1) > 0$ for all $x \geq 1$.
\end{proof}

\begin{lemma}\label{g(x)_lower_bound}
    Let $x \geq 1$ be a variable and $\lvert a \rvert,\lvert b \rvert \leq 0.01$ be constants. Define 
    \[g(x) = \frac{(x+a)(x+1)+2b+2\sqrt{x^3+ax+b}\sqrt{1+a+b}}{(1-x)^2}.\]
    Then $g(x) \geq 1$ for all $x \geq 1$.
\end{lemma}
\begin{proof}
    Note that 
    \[(x+a)(x+1)+2b \geq x^2+0.99x+0.97 \geq x^2-2x+1 = (x-1)^2.\]
    Therefore, 
    \[(x+a)(x+1)+2b+2\sqrt{x^3+ax+b}\sqrt{1+a+b} \geq (x+a)(x+1)+2b \geq (x-1)^2.\]
\end{proof}

\begin{lemma}\label{g(x)_upper_bound}
    Let $x \geq 1$ be a variable and $\lvert a \rvert,\lvert b \rvert \leq 0.01$ be constants. Define 
    \[g(x) = \frac{(x+a)(x+1)+2b+2\sqrt{x^3+ax+b}\sqrt{1+a+b}}{(1-x)^2}.\]
    Then $g(x) \leq \frac{(2x+1)^2}{(x-1)^2}$ for all $x \geq 1$.
\end{lemma}
\begin{proof}
    Note that 
    \[x^3+ax+b \leq x^3+x+0.25 \leq x^4+x^2+0.25 = (x^2+0.5)^2.\]
    Therefore, 
    \[2\sqrt{x^3+ax+b}\sqrt{1+a+b} \leq 2\sqrt{1.02}(x^2+0.5) \leq 3(x^2+0.5).\]
    It follows that 
    \begin{align*}
        (x+a)(x+1)+2b+2\sqrt{x^3+ax+b}\sqrt{1+a+b} &\leq x^2+1.01x+0.03+3(x^2+0.5) \\
        &\leq 4x^2+2x+2 \leq 4x^2+4x+1 \\
        &= (2x+1)^2.
    \end{align*}
\end{proof}

\section{Table for Proposition~\ref{MS_bound}}

We have the following table of $\cos \theta$ and 
\[E(\theta) = \exp\left(\frac{1+\sin \theta}{2\sin \theta}\log\frac{1+\sin \theta}{2\sin \theta} - \frac{1-\sin \theta}{2\sin \theta}\log \frac{1-\sin \theta}{2\sin \theta}+0.001\right)\]
for $n=2,\ldots,20$.
\begin{table}[h]
\centering
\begin{tabular}{|r|c|c|}
\hline
$n$ & $\cos\theta$ & $E(\theta)$ \\
\hline
2  & 0.9295160031 & 3.6029265222 \\
3  & 0.8000000000 & 2.1186523293 \\
4  & 0.7333333333 & 1.8270722583 \\
5  & 0.6909090909 & 1.6930091121 \\
6  & 0.6615384615 & 1.6154645667 \\
7  & 0.6400000000 & 1.5648147297 \\
8  & 0.6235294118 & 1.5291061421 \\
9  & 0.6105263158 & 1.5025674344 \\
10 & 0.6000000000 & 1.4820645884 \\
11 & 0.5913043478 & 1.4657471511 \\
12 & 0.5840000000 & 1.4524515355 \\
13 & 0.5777777778 & 1.4414091501 \\
14 & 0.5724137931 & 1.4320918024 \\
15 & 0.5677419355 & 1.4241244810 \\
16 & 0.5636363636 & 1.4172335583 \\
17 & 0.5600000000 & 1.4112146878 \\
18 & 0.5567567568 & 1.4059121773 \\
19 & 0.5538461538 & 1.4012053190 \\
20 & 0.5512195122 & 1.3969990839 \\
\hline
\end{tabular}
\end{table}

\end{document}